\documentclass[11pt,leqno]{amsart}

\usepackage[T1]{fontenc}
\usepackage{lmodern}
\usepackage{amsmath,amssymb,amsthm,mathtools}
\usepackage[hidelinks]{hyperref}
\usepackage{microtype}

\newtheorem{theorem}{Theorem}[section]
\newtheorem{proposition}[theorem]{Proposition}
\newtheorem{lemma}[theorem]{Lemma}
\newtheorem{corollary}[theorem]{Corollary}
\newtheorem{remark}[theorem]{Remark}

\newcommand{\F}{\mathbb F}
\newcommand{\Tr}{\operatorname{Tr}}
\newcommand{\Aut}{\operatorname{Aut}}
\newcommand{\Gal}{\operatorname{Gal}}
\newcommand{\Jac}{\operatorname{Jac}}
\newcommand{\NP}{\operatorname{NP}}
\newcommand{\cR}{\mathcal R}
\newcommand{\cY}{\mathcal Y}
\newcommand{\cZ}{\mathcal Z}
\newcommand{\wpmap}{\wp}

\title[On the generalized Ree curve]
{On the generalized Ree curve}

\author{Ahmad Kazemifard}
\address{Department of Mathematics, Faculty of Mathematical Sciences and Computer,
Shahid Chamran University of Ahvaz, Ahvaz, Iran}
\email{a.kazemifard@scu.ac.ir}

\author{Saeed Tafazolian}
\address{IMECC/UNICAMP, R. S\'ergio Buarque de Holanda, 651,
Cidade Universit\'aria ``Zeferino Vaz'', 13083-859 Campinas, SP, Brazil}
\email{tafazolian@ime.unicamp.br}


\begin{document}

\begin{abstract}
Let $p$ be an odd prime, let $q_0:=p^s$ with $s\ge1$, and put
$q:=pq_0^2$.  We study the smooth projective curve $\cR_{p,s}$ with
function field
\[
 \F_q(x,y,z),\qquad
 y^q-y=x^{q_0}(x^q-x),\qquad
 z^q-z=x^{2q_0}(x^q-x).
\]
The two-equation layer is already present in the ray-class and
big-action literature; for $p=3$ it gives the classical Ree curve.
In analogy with the generalized Suzuki curve, we call $\cR_{p,s}$
the \emph{generalized Ree curve} associated with $p$ and $s$.

For $p>3$ our principal geometric result determines the full geometric
automorphism group:
\[
 \Aut_{\overline{\F}_q}(\cR_{p,s})
 =\Aut_{\F_q}(\cR_{p,s})
 \cong U\rtimes\F_q^\times,
 \qquad |U|=q^3.
\]
The natural $p$-group $U$ gives a big action; we determine the complete
ramification filtration of the elementary abelian cover
$\cR_{p,s}\to\mathbb P^1_x$ and prove that
$(\cR_{p,s},P_\infty)$ is Castle.  In characteristic $3$ the same
subgroup is only the stabilizer of $P_\infty$ in the full Ree group,
so the automorphism structure exhibits a sharp characteristic-three
dichotomy.

On the arithmetic side, writing
\[
 \cY_{p,s}: y^q-y=x^{q_0}(x^q-x),\qquad
 \cZ_{p,s}: z^q-z=x^{2q_0}(x^q-x),
\]
we prove the $\F_q$-isogeny
\[
 \Jac(\cR_{p,s})\sim
 \Jac(\cY_{p,s})\times\Jac(\cZ_{p,s})^q.
\]
The first factor is supersingular and all three curves have $p$-rank
zero.  For $p>3$ the first Newton slope of $\cZ_{p,s}$ satisfies
\[
 \frac1{pq_0+2}\le \lambda_{\min}(\cZ_{p,s})\le\frac13,
\]
so $\cZ_{p,s}$ and $\cR_{p,s}$ are not supersingular.  For fixed
$p>3$ and $s\to\infty$, the order of the full automorphism group is
asymptotic to $2^{8/5}p^{-4/5}g^{8/5}$.
\end{abstract}

\maketitle

\section{Introduction}

Let $p$ be an odd prime and let
\[
 q_0:=p^s,\qquad q:=pq_0^2=p^{2s+1},\qquad s\ge1.
\]
We consider the smooth projective model $\cR_{p,s}$ of
\begin{equation}\label{eq:family}
 \begin{cases}
 y^q-y=x^{q_0}(x^q-x),\\
 z^q-z=x^{2q_0}(x^q-x).
 \end{cases}
\end{equation}
For $p=3$ these are the standard equations of the Ree curve.  The
classical Ree curve is arithmetically exceptional in two related senses.
If $g_{\rm Ree}$ denotes its genus and
\[
 N_q(g):=\max\{\#C(\F_q): C/\F_q\text{ a smooth projective curve of genus }g\},
\]
then
\[
 \#\cR_{3,s}(\F_q)=q^3+1=N_q(g_{\rm Ree}),
\]
so it is optimal over $\F_q$ in the standard sense; moreover it is
$\F_{q^6}$-maximal.  See Fuhrmann--Torres \cite{FT} for the terminology
and \cite[Introduction and Proposition~3.1]{TT} for the Ree curve.
In particular, it is supersingular.  We recall the maximal/minimal
terminology and this implication in Section~\ref{sec:prelim}.

A principal motivation for the present work is the odd-characteristic
extension of the Suzuki picture developed by Borges and Coutinho
\cite{BC}.  Starting from the Suzuki-type Artin--Schreier equation, they
determine rational point counts, the $L$-polynomial and the full
automorphism group of the generalized Suzuki curve.  One of the two
quotients of \eqref{eq:family}, denoted below by $\cY_{p,s}$, is in fact
that generalized Suzuki function field after the parameter shift
$t=s+1$; we record the explicit identification in
Remark~\ref{rem:Y=GS}.  Our question concerns the \emph{two-equation}
Ree layer and, especially, its full automorphism group.

The defining family itself is not new.  It occurs among Auer's ray-class
fields \cite{Auer}.  More precisely, after replacing the parameter $s$
in Matignon--Rocher \cite[Proposition~6.10]{MR} by $s+1$, the function
field of $\cR_{p,s}$ is their two-generator layer $L_{2,0}$.  The same
two Artin--Schreier equations also occur as the initial Ree-type layer
inside the larger covering constructed by Chr\'etien--Matignon
\cite[Section~3]{CM} to produce big actions with nonabelian derived
subgroup.  Thus no novelty is claimed for the equations or for the
existence of the natural translation action.  The principal question
here is structural and automorphism-theoretic: how much of the
exceptional symmetry of the classical Ree curve survives when this
known ray-class layer is placed in another odd characteristic?  We
determine its complete wild ramification structure and, for $p>3$, its
full geometric automorphism group, and compare the result with the
characteristic-three case.

Borges, Niemann and Zini \cite[Sections~2--5]{BNZ} study double
Artin--Schreier extensions of rational function fields and the lifting
of automorphisms from the rational base.  Their full-group results
concern the Singer, Zieve and extended Zieve families.  Their
preliminaries also include a one-pole lemma giving total ramification
and $p$-rank zero, and the classical genus-sum method for degree-$p$
subextensions.  We identify the application of the one-pole lemma in
Section~\ref{sec:basic} and record the common genus identity in
Section~\ref{sec:prelim}.  The explicit ramification filtration,
Jacobian factor grouping and full geometric automorphism group are
computed here for the Ree-type ray-class layer.

The natural subgroup $U$ of order $q^3$ is a big action, and the
unique point $P_\infty$ above the pole of $x$ makes the curve Castle.
Our main geometric result is the following.

\begin{theorem}\label{thm:intro-aut}
For every odd prime $p$, the pointed curve $(\cR_{p,s},P_\infty)$ is a
Castle curve.  If $p>3$, then
\[
 \Aut_{\overline{\F}_q}(\cR_{p,s})
 =\Aut_{\F_q}(\cR_{p,s})
 \cong U\rtimes\F_q^\times,
 \qquad |U|=q^3,
\]
and hence
\[
 |\Aut_{\overline{\F}_q}(\cR_{p,s})|=q^3(q-1).
\]
\end{theorem}

For fixed $p>3$ this group is genuinely large in the genus scale:
\[
 |\Aut(\cR_{p,s})|=\Theta(g^{8/5}),\qquad
 |U|=\Theta(g^{6/5}),
\]
so $|\Aut(\cR_{p,s})|/g\to\infty$ as $s\to\infty$.  For $p=3$, the
subgroup in Theorem~\ref{thm:intro-aut} is only the stabilizer of
$P_\infty$ in the Ree group, whose order is
$q^3(q-1)(q^3+1)$.

This jump has a conceptual explanation in the classification of
Giulietti--Korchm\'aros \cite[Theorem~1.3]{GK}.  The first ramification
group at $P_\infty$ contains $U$, and our estimate $|U|/g>q_0$ places
it above their large-wild-stabilizer threshold.  For $p>3$ the
exceptional alternatives in that theorem are excluded, and the full
automorphism group lies in the fixed-point branch.  When $p=3$, by
contrast, the curve is precisely the exceptional Deligne--Lusztig Ree
curve.  Thus the general classification predicts the same
characteristic-three dichotomy that appears in the explicit group
calculation.

The arithmetic geometry complements this automorphism-theoretic
picture.  Set
\begin{equation}\label{eq:Y-Z}
 \cY_{p,s}:\quad y^q-y=x^{q_0}(x^q-x),
 \qquad
 \cZ_{p,s}:\quad z^q-z=x^{2q_0}(x^q-x).
\end{equation}
The two quotients account for the whole Jacobian.

\begin{theorem}\label{thm:intro-decomp}
There is an isogeny over $\F_q$
\[
 \Jac(\cR_{p,s})
 \sim_{\F_q}
 \Jac(\cY_{p,s})\times\Jac(\cZ_{p,s})^q.
\]
Moreover, $\Jac(\cY_{p,s})$ is geometrically supersingular.
Consequently
\[
 \NP(\cR_{p,s})
 =\{1/2\}^{(q-1)pq_0}
 \sqcup q\,\NP(\cZ_{p,s}),
\]
where slopes are counted with multiplicity.
\end{theorem}

The same Artin--Schreier description makes the $p$-rank and wild
ramification completely explicit.

\begin{theorem}\label{thm:intro-prank}
For every odd prime $p$,
\[
 \gamma(\cR_{p,s})=\gamma(\cY_{p,s})=\gamma(\cZ_{p,s})=0.
\]
Let
\[
 V:=\Gal\bigl(\F_q(\cR_{p,s})/\F_q(x)\bigr)\cong(\F_q^2,+)
\]
and let $Z\le V$ be the subgroup acting trivially on $y$.  Then the
upper ramification jumps of $V$ at $P_\infty$ are $P+1$ and $P+2$,
where $P=pq_0$, while the lower jumps are $P+1$ and $P+q+1$.
\end{theorem}

The Newton-polygon calculation is used here mainly to distinguish the
arithmetic of the generalized curve from the classical Ree case.  The
degree-$p$ quotients of $\cZ_{p,s}$ have reduced pole order $pq_0+2$;
Zhu's Newton-over-Hodge theorem \cite{Zhu} gives the lower bound below,
while three explicit point counts give the upper bound.

\begin{theorem}\label{thm:intro-NP}
Let $p>3$.  Then
\[
 \frac1{pq_0+2}
 \le \lambda_{\min}(\cZ_{p,s})
 \le \frac13.
\]
In particular, $\cZ_{p,s}$ and $\cR_{p,s}$ are not supersingular and
$\cR_{p,s}$ is neither maximal nor minimal over any finite extension
of $\F_q$.
\end{theorem}

Thus characteristic $3$ is exceptional in two distinct but compatible
ways: the classical Ree curve is supersingular and has the full Ree
group, whereas for $p>3$ the generalized Ree curve is
non-supersingular and its full geometric automorphism group is the
point stabilizer described above.

We finish with an application of a construction of Voloch
\cite{Voloch}.  A useful feature of the present family is
\[
 \cR_{p,s}(\F_{q^2})=\cR_{p,s}(\F_q).
\]
This produces an unramified abelian cover over $\F_{q^2}$ of degree
$L_{\cR}(-1)$ in which all $\F_{q^2}$-rational points of $\cR_{p,s}$
split completely.  The Jacobian decomposition above gives the factor
$L_{\cR}(-1)=L_{\cY}(-1)L_{\cZ}(-1)^q$.

\section{Preliminaries}\label{sec:prelim}

\subsection{Elementary abelian Artin--Schreier extensions}

For a field $K$ of characteristic $p$ we write
\[
 \wpmap(h):=h^p-h.
\]
The class $[f]\in K/\wpmap(K)$ determines an Artin--Schreier
character once the Galois group is identified with $\F_p$.  For
nonzero classes, the degree-$p$ fields
\[
 K(u),\quad u^p-u=f,
 \qquad\text{and}\qquad
 K(v),\quad v^p-v=g,
\]
coincide if and only if
\begin{equation}\label{eq:AS-orbits}
 [g]=a[f]\quad\text{in }K/\wpmap(K)
 \qquad\text{for some }a\in\F_p^\times.
\end{equation}
Indeed, the forward implication follows by choosing a generator
$\sigma$ of the common cyclic Galois group: after normalizing
$\sigma(u)-u=1$, one has $\sigma(v)-v=a\in\F_p^\times$, whence
$v-au\in K$ and therefore $g-af\in\wpmap(K)$.  The converse is
immediate.  Thus degree-$p$ Artin--Schreier extensions correspond to
one-dimensional $\F_p$-subspaces of $K/\wpmap(K)$.

For an algebraically closed field $k$ and $f\in k[x]$, its class modulo
$\wpmap(k[x])$ has a unique reduced polynomial representative after
the constant term is normalized to zero, where ``reduced'' means that
no nonconstant exponent is divisible by $p$; see
Matignon--Rocher~\cite[p.~890]{MR}.  If
\[
 C:\quad w^p-w=f(x)
\]
has reduced polynomial $f$ of degree $d$ with $p\nmid d$, then $C$
has a unique point above infinity and
\begin{equation}\label{eq:AS-genus}
 g(C)=\frac{(p-1)(d-1)}2.
\end{equation}
The conductor exponent of the corresponding character at infinity is
$d+1$; see \cite[Chapter III, \S7]{Stichtenoth}.

We shall repeatedly use the factorization attached to an elementary
abelian $p$-extension.  Let $E/K$ be such an extension, and let $E_H$
run over the degree-$p$ subextensions of $E/K$.  When $K$ is rational,
Duursma--Stichtenoth--Voss \cite[Corollary~6.7]{DSV} give the
$L$-polynomial factorization
\begin{equation}\label{eq:EA-factor}
 L_E(T)=\prod_H L_{E_H}(T).
\end{equation}
Taking degrees, using $\deg L_F=2g(F)$ for a function field $F$ with
full constant field, yields
\begin{equation}\label{eq:EA-genus}
 g(E)=\sum_H g(E_H).
\end{equation}
The corresponding Jacobian decomposition
\begin{equation}\label{eq:KR-isog}
 \Jac(E)\sim \prod_H\Jac(E_H)
\end{equation}
is a special case of the idempotent relations of Kani--Rosen
\cite{KR}.  The rationality of the base is essential for the simple
factorization \eqref{eq:EA-factor}.
For double Artin--Schreier extensions, the same genus-sum method is
used by Borges, Niemann and Zini \cite[Eq.~(2.4)]{BNZ}, following
Garcia--Stichtenoth \cite[Theorem~2.1]{GS}.  Here the $p$-ranks are
computed directly by the Deuring--Shafarevich formula
\eqref{eq:DS}.

\subsection{Maximal, minimal and supersingular curves}\label{sec:maxmin}

Let $C/\F_q$ be a smooth projective curve of genus $g$.  The Hasse--Weil
bound is
\[
 \bigl|\#C(\F_q)-(q+1)\bigr|\le 2g\sqrt q.
\]
The curve is called \emph{$\F_q$-maximal} (respectively
\emph{$\F_q$-minimal}) when equality holds with the upper (respectively
lower) sign.  More generally, one uses the same terminology after
extending constants to $\F_{q^n}$.  This is the convention used, for
example, in Borges--Coutinho \cite[\S2.1]{BC}.

Write
\[
 L_C(T)=\prod_{j=1}^{2g}(1-\alpha_jT),
 \qquad |\alpha_j|=\sqrt q.
\]
If $C$ is maximal over $\F_{q^n}$, then
\[
 \sum_{j=1}^{2g}\alpha_j^n=-2gq^{n/2};
\]
if it is minimal, the right-hand side is $+2gq^{n/2}$.  Equality in the
triangle inequality forces, respectively,
\[
 \alpha_j^n=-q^{n/2}\quad\text{for all }j,
 \qquad\text{or}\qquad
 \alpha_j^n=q^{n/2}\quad\text{for all }j.
\]
Thus every normalized Frobenius eigenvalue $\alpha_j/\sqrt q$ is a root
of unity.  In particular,
\begin{equation}\label{eq:maxmin-ss}
 C\text{ maximal or minimal over some finite extension}
 \quad\Longrightarrow\quad
 \Jac(C)\text{ is supersingular}.
\end{equation}
Equivalently, all Newton slopes are $1/2$.  We shall use the
contrapositive of \eqref{eq:maxmin-ss} in Section~\ref{sec:NP}.

\subsection{Supersingular curves of the form \texorpdfstring{$w^p-w=xR(x)$}{Artin--Schreier form}}

Let $R(x)\in\F_{p^m}[x]$ be a $p$-linearized polynomial of degree
$p^h$ with $0\le h\le m/2$.  Van der Geer and van der Vlugt prove that
the curve
\begin{equation}\label{eq:vGvV}
 w^p-w=xR(x)
\end{equation}
has Jacobian isogenous, over $\F_{p^m}$, to a product of supersingular
elliptic curves; see \cite[Theorem~13.7]{vGvV}.  In applications below
we pass from $\F_q$ to $\F_{q^2}$ in order to meet the numerical
condition $h\le m/2$.  Supersingularity is geometric and is therefore
unchanged by this base extension.

\subsection{Newton polygons and the Hodge lower bound}

For a smooth projective curve $C/\F_q$ of genus $g$ write
\[
 L_C(T)=\prod_{j=1}^{2g}(1-\alpha_jT)
       =1+c_1T+\cdots+c_{2g}T^{2g}.
\]
We normalize the $p$-adic valuation by $v_q(q):=1$.  The Newton
polygon of $C$ is the lower convex hull of the points
$(i,v_q(c_i))$.  Its slopes are the numbers $v_q(\alpha_j)$, counted
with multiplicity.  The curve is supersingular precisely when all
slopes are $1/2$.

We shall use Zhu's Newton-over-Hodge theorem \cite{Zhu}.  In the
one-pole polynomial case it implies the following: if
\[
 w^p-w=f(x),\qquad \deg f=d,\qquad p\nmid d,
\]
then the first Newton slope is at least $1/d$.  More precisely, the
Newton polygon lies above the Hodge polygon with slopes
$1/d,2/d,\ldots,(d-1)/d$ for each nontrivial additive character.

\subsection{The Deuring--Shafarevich formula}

Let $C$ be a curve over an algebraically closed field of
characteristic $p$, let $S$ be a finite $p$-subgroup of
$\Aut(C)$, and let $\gamma(C)$ denote the $p$-rank of $C$.  If
$\ell_1,\ldots,\ell_r$ are the lengths of the short orbits of $S$,
the Deuring--Shafarevich formula (see Korchm\'aros--Montanucci \cite[Eq.~(5)]{KM2019}) reads
\begin{equation}\label{eq:DS}
 \gamma(C)-1
 =
 |S|\bigl(\gamma(C/S)-1\bigr)
 +\sum_{i=1}^r(|S|-\ell_i).
\end{equation}
We use it below only in the particularly simple situation where
$C/S$ is rational and $S$ has a unique short orbit, consisting of one
totally ramified point.

\subsection{Castle curves}

Let $C/\F_q$ be a curve and $Q\in C(\F_q)$.  Put
\[
 m(Q):=\min(H(Q)\setminus\{0\}),
\]
where $H(Q)$ is the Weierstrass semigroup at $Q$.  Following
Munuera--Sep\'ulveda--Torres \cite{Castle}, the pointed curve $(C,Q)$
is Castle if $H(Q)$ is symmetric and
\[
 \#C(\F_q)=q\,m(Q)+1.
\]
We shall use the Lewittes bound
\begin{equation}\label{eq:Lewittes}
 \#C(\F_q)\le q\,m(Q)+1.
\end{equation}
We also use the standard criterion
\[
 H(Q)\text{ symmetric}
 \quad\Longleftrightarrow\quad
 (2g-2)Q\text{ is canonical}.
\]

\subsection{Big actions}

Let $k$ be algebraically closed of characteristic $p$, let $C/k$ have
genus at least two, and let $S$ be a $p$-subgroup of $\Aut_k(C)$.  The
pair $(C,S)$ is a big action if
\[
 \frac{|S|}{g(C)}>\frac{2p}{p-1}.
\]
We use Corollary~2.11 of Matignon--Rocher \cite{MR}: if
$A_{\infty,1}$ is the full wild inertia group at the unique ramified
point of a big action, then
\begin{equation}\label{eq:MR-derived}
 D(A_{\infty,1})=D(S).
\end{equation}
We shall also use Theorem~1.3 of Giulietti--Korchm\'aros \cite{GK}.
If $G\le\Aut_k(C)$ and $P\in C$ satisfy
\[
 |G_P^{(1)}|>\frac{p}{p-1}g(C),
\]
then either $G$ fixes $P$, or $C$ belongs to one of the four exceptional
families listed there: the odd hyperelliptic case, the Hermitian curve,
the Suzuki curve, or the Ree curve.

\section{The ray-class layer and the Castle property}\label{sec:basic}

Put
\[
 P:=pq_0=p^{s+1}.
\]
Since $q_0P=q$, for $a\in\F_q$ we write
\[
 a^{1/q_0}:=a^P;
\]
this is the inverse of the $q_0$-power Frobenius on $\F_q$.
Let $L:=\F_q(x,y,z)$ be the function field of $\cR_{p,s}$.
The equations \eqref{eq:family} are the first two equations of the
odd-characteristic ray-class family described by Auer \cite{Auer}.
With the parameter of Matignon--Rocher shifted from $s$ to $s+1$,
they are exactly the equations defining the layer $L_{2,0}$ in
\cite[Proposition~6.10]{MR}.  They also occur in the larger Ree-type
covering of Chr\'etien--Matignon \cite[Section~3]{CM}.

\begin{proposition}\label{prop:degree-p}
The extension $L/\F_q(x)$ is elementary abelian of degree $q^2$.
Its degree-$p$ quotients are represented by
\begin{equation}\label{eq:F-lambda-mu}
 F_{\lambda,\mu}(x):=
 \lambda^{1/q_0}x^{P+1}-\lambda x^{q_0+1}
 +\mu^{1/q_0}x^{P+2}-\mu x^{2q_0+1},
\end{equation}
where $(\lambda,\mu)\in\F_q^2\setminus\{(0,0)\}$, modulo multiplication
by $\F_p^\times$.
\end{proposition}

\begin{proof}
Put $K:=\F_q(x)$ and
\[
 f:=x^{q_0}(x^q-x),\qquad
 g:=x^{2q_0}(x^q-x),
\]
so that $y^q-y=f$ and $z^q-z=g$.  Write
$m:=[\F_q:\F_p]=2s+1$.  We first prove the degree assertion without
using the character group of $L/K$.

For $(\lambda,\mu)\in\F_q^2$ define, inside $L$,
\[
 w_{\lambda,\mu}
 :=\sum_{i=0}^{m-1}(\lambda y+\mu z)^{p^i}.
\]
Telescoping and $\lambda^q=\lambda$, $\mu^q=\mu$ give
\begin{equation}\label{eq:w-lambda-mu}
 w_{\lambda,\mu}^p-w_{\lambda,\mu}
 =\lambda f+\mu g.
\end{equation}
Consider the $\F_p$-linear map
\[
 \Phi:\F_q^2\longrightarrow K/\wpmap(K),\qquad
 (\lambda,\mu)\longmapsto[\lambda f+\mu g].
\]
We claim that $\Phi$ is injective.  Since
\[
 q+q_0=q_0(P+1),\qquad q+2q_0=q_0(P+2),
\]
repeated Artin--Schreier reduction of the $q_0$-power exponents gives
\begin{equation}\label{eq:reduction-Phi}
 [\lambda f+\mu g]
 = [F_{\lambda,\mu}(x)]
 \quad\text{in }K/\wpmap(K),
\end{equation}
where $F_{\lambda,\mu}$ is the polynomial in
\eqref{eq:F-lambda-mu}.  All four nonconstant exponents occurring in
$F_{\lambda,\mu}$ are prime to $p$.  If $\mu\ne0$, its degree is
$P+2$; if $\mu=0$ and $\lambda\ne0$, its degree is $P+1$.

A nonzero polynomial $F\in\F_q[x]$ of positive degree $d$ with
$p\nmid d$ cannot lie in $\wpmap(K)$.  Indeed, if $F=h^p-h$ with
$h\in K$, then $h$ cannot have a finite pole, since such a pole would
remain a pole of $h^p-h$; hence $h\in\F_q[x]$.  If $h$ is
nonconstant, then $\deg(h^p-h)=p\deg h$, contrary to $p\nmid d$;
if $h$ is constant, then $F$ is constant.  Thus
$[F_{\lambda,\mu}]\ne0$ for every
$(\lambda,\mu)\ne(0,0)$, proving the injectivity of $\Phi$.

Hence the image of $\Phi$ is an $\F_p$-subspace of
$K/\wpmap(K)$ of dimension $2m$.  By elementary abelian
Artin--Schreier theory; see \cite[Propositions~1.1 and~1.2]{GS}, the compositum $M/K$ of the corresponding
cyclic degree-$p$ extensions has degree
\[
 [M:K]=p^{2m}=q^2.
\]
Equation \eqref{eq:w-lambda-mu} shows that all these cyclic extensions
are contained in $L$, so $M\subseteq L$.  On the other hand, the two
displayed equations for $y$ and $z$ give
\[
 [L:K]\le q^2.
\]
Therefore
\[
 q^2=[M:K]\le[L:K]\le q^2,
\]
so $L=M$ and $[L:K]=q^2$.

For every $(c,d)\in\F_q^2$, the translation
\[
 \tau_{c,d}:(x,y,z)\longmapsto(x,y+c,z+d)
\]
is a $K$-automorphism of $L$.  These $q^2$ automorphisms are distinct,
and since $[L:K]=q^2$ they constitute the full Galois group.  Thus
\[
 V:=\Gal(L/K)\cong(\F_q^2,+).
\]
Only now do we use the trace pairing.  Every character of $V$ has the
form
\[
 \chi_{\lambda,\mu}(\tau_{c,d})
 :=\Tr_{\F_q/\F_p}(\lambda c+\mu d),
 \qquad (\lambda,\mu)\in\F_q^2,
\]
and the trace pairing makes $(\lambda,\mu)\mapsto\chi_{\lambda,\mu}$
an isomorphism $\F_q^2\simeq V^\vee$.  Moreover,
\[
 \tau_{c,d}(w_{\lambda,\mu})-w_{\lambda,\mu}
 =\chi_{\lambda,\mu}(\tau_{c,d}),
\]
so the fixed field of $\ker\chi_{\lambda,\mu}$ is generated over $K$
by $w_{\lambda,\mu}$ and has the Artin--Schreier equation
\eqref{eq:w-lambda-mu}, whose reduced right-hand side is
$F_{\lambda,\mu}$.  Finally, two nonzero characters have the same
kernel exactly when they differ by an element of $\F_p^\times$.
This gives precisely the asserted parametrization of the degree-$p$
quotients.
\end{proof}

The total ramification assertion in Corollary~\ref{cor:genus-points}
and the $\cR_{p,s}$ case of Proposition~\ref{prop:prank} are instances
of Borges--Niemann--Zini \cite[Lemma~2.1]{BNZ}.  To check the
hypotheses, work over $k:=\overline{\F}_q$ and put
\[
 u:=x^{q_0}(x^q-x),\qquad
 v:=x^{2q_0}(x^q-x),\qquad r:=v/u.
\]
Then $r=x^{q_0}$ and
\[
 x=r^P-\frac{u}{r},
\]
so $k(u,v)=k(x)$ and $k(x,y,z)=k(y,z)$.  The reduced-polynomial
argument of Proposition~\ref{prop:degree-p} remains valid over $k$;
hence $[k(y,z):k(u,v)]=q^2$.  Both $u$ and $v$ have infinity as
their unique pole.  The cited lemma therefore gives a unique point
above infinity, total ramification there, no other ramification, and
$\gamma(\cR_{p,s})=0$.  We retain the direct proofs below.

\begin{corollary}\label{cor:genus-points}
The curve $\cR_{p,s}$ has a unique point $P_\infty$ above the pole of
$x$, and
\begin{equation}\label{eq:genus-R}
 g(\cR_{p,s})
 =\frac{q-1}{2}\bigl[P+q(P+1)\bigr].
\end{equation}
Moreover,
\begin{equation}\label{eq:R-Fq}
 \#\cR_{p,s}(\F_q)=q^3+1.
\end{equation}
The two quotient curves in \eqref{eq:Y-Z} have genera
\begin{equation}\label{eq:gYgZ}
 g(\cY_{p,s})=\frac{(q-1)P}{2},\qquad
 g(\cZ_{p,s})=\frac{(q-1)(P+1)}{2}.
\end{equation}
\end{corollary}

\begin{proof}
Every nontrivial degree-$p$ quotient is ramified only at infinity.
Hence the full extension is unramified at every finite place.  At
infinity every nonzero character is ramified, so the inertia subgroup
is contained in no character kernel; since $V$ is elementary abelian,
this forces the inertia subgroup to be all of $V$.  Thus infinity is
totally ramified and there is a unique point $P_\infty$ above it.
There are $(q-1)/(p-1)$ character lines with $\mu=0$ and
$q(q-1)/(p-1)$ with $\mu\ne0$.  Formula \eqref{eq:AS-genus} and the genus identity
\eqref{eq:EA-genus} give \eqref{eq:genus-R} and \eqref{eq:gYgZ}.
For $x\in\F_q$ both right-hand sides in \eqref{eq:family} vanish.
Thus there are $q^3$ affine rational points, together with
$P_\infty$.
\end{proof}

We now make the ramification at infinity explicit.  Let
\[
 V:=\Gal(L/\F_q(x))
   =\{\tau_{c,d}:(x,y,z)\mapsto(x,y+c,z+d):c,d\in\F_q\}
\]
and put
\[
 Z:=\{\tau_{0,d}:d\in\F_q\}.
\]

\begin{proposition}\label{prop:ramification}
The upper ramification filtration of $V$ at $P_\infty$ is
\[
 V^u=
 \begin{cases}
 V,&0\le u\le P+1,\\
 Z,&P+1<u\le P+2,\\
 1,&u>P+2.
 \end{cases}
\]
Equivalently, the lower ramification filtration is
\[
 V_i=
 \begin{cases}
 V,&0\le i\le P+1,\\
 Z,&P+2\le i\le P+q+1,\\
 1,&i\ge P+q+2.
 \end{cases}
\]
In particular, the different exponent at $P_\infty$ is
\begin{equation}\label{eq:different-explicit}
 d_\infty=(P+2)(q^2-1)+q(q-1).
\end{equation}
\end{proposition}

\begin{proof}
The nontrivial characters of $V$ are indexed by the lines
$[\lambda:\mu]$ in Proposition~\ref{prop:degree-p}.  Their conductor
exponents are $P+2$ when $\mu=0$ and $P+3$ when $\mu\ne0$.
For an abelian local extension, the conductor exponent of a
nontrivial character is one plus its upper break.  Hence the
characters with $\mu=0$ have upper break $P+1$, whereas the remaining
characters have upper break $P+2$.  The intersection of the kernels
of all characters with $\mu=0$ is precisely $Z$, which gives the
upper filtration.

Herbrand's function then gives the second lower jump:
\[
 \psi(P+2)
 =(P+1)+[V:Z]\bigl((P+2)-(P+1)\bigr)
 =P+q+1.
\]
The formula for the different follows from
$d_\infty=\sum_{i\ge0}(|V_i|-1)$.
\end{proof}

\begin{proposition}\label{prop:prank}
For every odd prime $p$,
\[
 \gamma(\cR_{p,s})
 =\gamma(\cY_{p,s})
 =\gamma(\cZ_{p,s})=0.
\]
\end{proposition}

\begin{proof}
For $\cR_{p,s}$, the $p$-group $V$ has quotient
$\mathbb P^1_x$, and its only short orbit is the fixed point
$P_\infty$.  Hence \eqref{eq:DS} gives
\[
 \gamma(\cR_{p,s})-1
 =-q^2+(q^2-1)=-1.
\]
The same argument applies to the degree-$q$ Galois covers
$\cY_{p,s}\to\mathbb P^1_x$ and
$\cZ_{p,s}\to\mathbb P^1_x$: in each case infinity is the unique
ramified point and is totally ramified.  Thus both $p$-ranks are
zero.
\end{proof}

\begin{corollary}\label{cor:zero-prank-nonss}
If $p>3$, then $\cR_{p,s}$ and $\cZ_{p,s}$ are $p$-rank-zero,
non-supersingular curves.
\end{corollary}

\begin{proof}
Combine Proposition~\ref{prop:prank} with
Theorem~\ref{thm:slope-bracket} and
Theorem~\ref{thm:decomp}.
\end{proof}

\begin{proposition}\label{prop:Castle}
For every odd prime $p$, the pointed curve
$(\cR_{p,s},P_\infty)$ is Castle.  More precisely,
\[
 m(P_\infty)=q^2
\]
and $H(P_\infty)$ is symmetric.
\end{proposition}

\begin{proof}
Since $[L:\F_q(x)]=q^2$ and $P_\infty$ is the unique point above the
pole of $x$,
\[
 (x)_\infty=q^2P_\infty.
\]
Hence $q^2\in H(P_\infty)$ and $m(P_\infty)\le q^2$.  On the other
hand, \eqref{eq:R-Fq} and the Lewittes bound \eqref{eq:Lewittes} give
\[
 q^3+1\le q\,m(P_\infty)+1,
\]
so $m(P_\infty)\ge q^2$.  Thus $m(P_\infty)=q^2$ and equality holds
in the Lewittes bound.

It remains to prove symmetry.  The extension $L/\F_q(x)$ is ramified
only at $P_\infty$, so its different is $\delta P_\infty$ for some
$\delta$.  Since the pole of $x$ is totally ramified with index
$q^2$,
\[
 (dx)_L=(\delta-2q^2)P_\infty.
\]
Riemann--Hurwitz gives $\delta-2q^2=2g-2$.  Consequently
\[
 (dx)_L=(2g-2)P_\infty,
\]
which is canonical.  Hence $H(P_\infty)$ is symmetric.
\end{proof}

We also record the natural automorphisms.  For
$a\in\F_q^\times$ and $b,c,d\in\F_q$ put
\begin{equation}\label{eq:psi}
\begin{aligned}
 \psi_{a,b,c,d}(x)&:=ax+b,\\
 \psi_{a,b,c,d}(y)&:=a^{q_0+1}y+ab^{q_0}x+c,\\
 \psi_{a,b,c,d}(z)&:=a^{2q_0+1}z
 +2a^{q_0+1}b^{q_0}y+ab^{2q_0}x+d.
\end{aligned}
\end{equation}
A direct substitution shows that these maps preserve both equations in
\eqref{eq:family}; composing two maps of the form \eqref{eq:psi} and
solving triangularly for the inverse shows that they form a group.
We denote it by $\Gamma$ and put
\[
 U:=\{\psi_{1,b,c,d}:b,c,d\in\F_q\}.
\]
Then
\begin{equation}\label{eq:Gamma-size}
 |U|=q^3,\qquad
 \Gamma=U\rtimes\F_q^\times,\qquad
 |\Gamma|=q^3(q-1).
\end{equation}

The group $U$ acts sharply transitively on the $q^3$ affine
$\F_q$-rational points of $\cR_{p,s}$ and fixes $P_\infty$.
Consequently $\Gamma$ has exactly two orbits on
$\cR_{p,s}(\F_q)$: the singleton $\{P_\infty\}$ and its affine
complement.

For later use, write $u(b,c,d):=\psi_{1,b,c,d}$ and regard these as
transformations of points.  We use the convention that products are
read from left to right: first apply $u(b,c,d)$ and then
$u(b',c',d')$.  Equivalently, the induced pull-backs on the function
field compose in the reverse order.  A direct calculation gives
\begin{equation}\label{eq:U-law}
\begin{aligned}
 u(b,c,d)u(b',c',d')
 =u(&b+b',\ c+c'+{b'}^{q_0}b,\\
 &d+d'+2{b'}^{q_0}c+{b'}^{2q_0}b).
\end{aligned}
\end{equation}

\section{The Jacobian decomposition}\label{sec:decomp}

For a nonzero pair $(\lambda,\mu)$ let $C_{[\lambda:\mu]}$ denote the
degree-$p$ quotient associated with the $\F_p$-line generated by
$(\lambda,\mu)$.  By \eqref{eq:EA-factor} and
\eqref{eq:KR-isog},
\begin{equation}\label{eq:R-allfactors}
 \Jac(\cR_{p,s})
 \sim_{\F_q}
 \prod_{[\lambda:\mu]}\Jac(C_{[\lambda:\mu]}).
\end{equation}
The factors with $\mu=0$ are precisely the factors occurring in
$\Jac(\cY_{p,s})$, while the factors $C_{[0:\mu]}$ are precisely those
occurring in $\Jac(\cZ_{p,s})$.

\begin{theorem}\label{thm:decomp}
There are factorizations
\begin{equation}\label{eq:L-decomp}
 L_{\cR}(T)=L_{\cY}(T)L_{\cZ}(T)^q
\end{equation}
and
\begin{equation}\label{eq:J-decomp}
 \Jac(\cR_{p,s})
 \sim_{\F_q}
 \Jac(\cY_{p,s})\times\Jac(\cZ_{p,s})^q.
\end{equation}
\end{theorem}

\begin{proof}
The assertion for the $\mu=0$ factors follows directly from the
factorization for the elementary abelian extension
$\F_q(x,y)/\F_q(x)$.  Consider now a line with $\mu\ne0$.
Put
\[
 u_b:=u(b,0,0),\qquad
 v_{c,d}:=u(0,c,d)\in V.
\]
Using \eqref{eq:U-law} and the inverse obtained from the same formula,
we find, with our point-action convention,
\[
 u_bv_{c,d}u_b^{-1}=v_{c,d-2b^{q_0}c}.
\]
Consequently conjugation by $u_b$ permutes the kernels of the
characters
\[
 \chi_{\lambda,\mu}(v_{c,d})
 :=\Tr_{\F_q/\F_p}(\lambda c+\mu d)
\]
according to
\begin{equation}\label{eq:character-action}
 [\lambda:\mu]\longmapsto
 [\lambda+2\mu b^{q_0}:\mu].
\end{equation}
Indeed, if $H_{\lambda,\mu}:=\ker(\chi_{\lambda,\mu})$, then
$u_bH_{\lambda,\mu}u_b^{-1}=H_{\lambda+2\mu b^{q_0},\mu}$, because
$u_b^{-1}v_{c,d}u_b=v_{c,d+2b^{q_0}c}$.  Since $p$ is odd and
$b\mapsto b^{q_0}$ is a permutation of $\F_q$, for fixed
$\mu\ne0$ the $q$ lines
\[
 [\lambda:\mu],\qquad \lambda\in\F_q,
\]
form a single orbit, represented by $[0:\mu]$.  Taking one $\mu$ from
each class in $\F_q^\times/\F_p^\times$, the product of the
representative factors is exactly the elementary abelian factorization
of $\Jac(\cZ_{p,s})$.  Grouping the factors in
\eqref{eq:R-allfactors} gives \eqref{eq:J-decomp}, and the
$L$-polynomial identity follows as well.
\end{proof}

\begin{remark}[The generalized Suzuki quotient]\label{rem:Y=GS}
The quotient $\cY_{p,s}$ is not a new curve.  Let $\mathcal X_{GS}$ be
the generalized Suzuki curve of Borges--Coutinho \cite{BC} with their
parameter $t=s+1$.  Thus it is defined over the same field $\F_q$ by
\[
 v^q-v=x^P(x^q-x),\qquad P=pq_0=p^{s+1}.
\]
The degree-$p$ quotients of this extension have reduced
Artin--Schreier representatives
\[
 G_\mu(x)=\mu^{1/P}x^{q_0+1}-\mu x^{P+1},
 \qquad \mu\in\F_q^\times.
\]
For the degree-$p$ quotients of $\cY_{p,s}$,
Proposition~\ref{prop:degree-p} gives
\[
 F_{\lambda,0}(x)=\lambda^{1/q_0}x^{P+1}
                   -\lambda x^{q_0+1}.
\]
Taking $\mu=-\lambda^{1/q_0}$ and using $Pq_0=q$ gives
$G_\mu=F_{\lambda,0}$.  Hence the two elementary abelian degree-$q$
extensions of $\F_q(x)$ have the same degree-$p$ subfields and therefore
coincide.  In particular, the point counts and $L$-polynomial of
$\cY_{p,s}$ are already determined in \cite{BC}.  We retain the
supersingularity argument below because it is intrinsic to the
Jacobian decomposition used here.
\end{remark}

\begin{proposition}\label{prop:Y-ss}
The Jacobian $\Jac(\cY_{p,s})$ is geometrically supersingular.
Consequently
\begin{equation}\label{eq:NP-decomp}
 \NP(\cR_{p,s})
 =\{1/2\}^{(q-1)P}\sqcup q\,\NP(\cZ_{p,s}).
\end{equation}
\end{proposition}

\begin{proof}
For $\lambda\in\F_q^\times$, the degree-$p$ quotient belonging to
$[\lambda:0]$ has equation
\begin{equation}\label{eq:C-lambda}
 w^p-w
 =\lambda^{1/q_0}x^{P+1}-\lambda x^{q_0+1}
 =xR_\lambda(x),
\end{equation}
where
\[
 R_\lambda(x):=\lambda^{1/q_0}x^P-\lambda x^{q_0}
\]
is $p$-linearized.  After extending constants to $\F_{q^2}$, the
hypotheses of \cite[Theorem~13.7]{vGvV} apply, because
$\deg R_\lambda=p^{s+1}$ and $s+1\le(4s+2)/2$.  Hence every
$\Jac(C_{[\lambda:0]})$ is geometrically supersingular.  The elementary
abelian factorization of $\Jac(\cY_{p,s})$ proves the first assertion.
Its dimension is $(q-1)P/2$, so its slope-$1/2$ contribution has
multiplicity $(q-1)P$.  Formula \eqref{eq:NP-decomp} now follows from
Theorem~\ref{thm:decomp}.
\end{proof}

\begin{remark}
Proposition~\ref{prop:Y-ss} separates the old supersingular part from
the new arithmetic.  In particular, for $p>3$ the failure of
supersingularity of $\cR_{p,s}$ is entirely caused by
$\Jac(\cZ_{p,s})$.
\end{remark}

\section{Point counts and the first Newton slope}\label{sec:NP}

For $n\ge1$ put $K_n:=\F_{q^n}$ and
\[
 M_n:=\#\left\{x\in K_n:
 \Tr_{K_n/\F_q}\bigl(x^{2q_0}(x^q-x)\bigr)=0\right\}.
\]
Since $Z^q-Z=a$ has $q$ solutions in $K_n$ precisely when
$\Tr_{K_n/\F_q}(a)=0$,
\begin{equation}\label{eq:Z-Mn}
 \#\cZ_{p,s}(K_n)=1+qM_n.
\end{equation}

\begin{proposition}\label{prop:three-counts}
For every odd $p$,
\begin{align}
 \#\cZ_{p,s}(\F_q)&=q^2+1,\label{eq:count1}\\
 \#\cZ_{p,s}(\F_{q^2})&=2q^2-q+1.\label{eq:count2}
\end{align}
If $p>3$, then
\begin{equation}\label{eq:count3}
 \#\cZ_{p,s}(\F_{q^3})=
 \begin{cases}
 q^3-q^2+q+1,&p\equiv1\pmod3,\\[1mm]
 q^3+q^2-q+1,&p\equiv2\pmod3.
 \end{cases}
\end{equation}
\end{proposition}

\begin{proof}
For $n=1$, $x^q-x=0$ for every $x\in\F_q$, so $M_1=q$.

For $n=2$, put $d:=x^q-x$.  Then $d^q=-d$ and
\[
\begin{aligned}
 \Tr_{K_2/\F_q}(x^{2q_0}d)
 &=d\bigl(x^{2q_0}-(x+d)^{2q_0}\bigr)\\
 &=-d^{q_0+1}(2x+d)^{q_0}.
\end{aligned}
\]
Thus the trace vanishes exactly when $d=0$ or $2x+d=0$.  The first
condition contributes $q$ elements, and the second is $x^q+x=0$ and
contributes $q-1$ additional nonzero elements.  Hence $M_2=2q-1$.

Assume now $p>3$ and put $K:=K_3$.  Let
\[
 \sigma(u):=u^q.
\]
Since the map $y\mapsto y^P$ is an automorphism of $K$, write
$x=y^P$.  Using $Pq_0=q$ we obtain
\[
 x^{2q_0}(x^q-x)
 =\sigma(y)^2(\sigma(y)-y)^P.
\]
Set
\[
 u:=y-\sigma^{-1}(y).
\]
Then $\sigma(y)-y=\sigma(u)$, and invariance of the relative trace
gives
\begin{equation}\label{eq:B-def}
 \Tr_{K/\F_q}\bigl(x^{2q_0}(x^q-x)\bigr)
 =B(y):=\Tr_{K/\F_q}(y^2u^P).
\end{equation}
The map $y\mapsto u$ has kernel $\F_q$ and image
$V:=\ker\Tr_{K/\F_q}$, with $|V|=q^2$.

Fix $u\in V$.  Its fibre consists of $y+c$, $c\in\F_q$.  Put
\[
 A(u):=\Tr_{K/\F_q}(yu^P).
\]
Since $\Tr(u^P)=0$,
\begin{equation}\label{eq:B-fibre}
 B(y+c)=B(y)+2cA(u).
\end{equation}
Hence if $A(u)\ne0$, exactly one point in the fibre satisfies $B=0$.

Suppose $u\ne0$ and set $t:=\sigma(u)/u$.  The equality
$\Tr(u)=0$ gives
\begin{equation}\label{eq:t-relation}
 1+t+t\sigma(t)=0.
\end{equation}
Because $p\ne3$, the element $y=(u-\sigma(u))/3$ satisfies
$y-\sigma^{-1}(y)=u$.  A direct calculation using
\eqref{eq:t-relation} yields
\begin{equation}\label{eq:Aformula}
 A(u)=u^{P+1}\bigl(1+t^P(1+t)\bigr).
\end{equation}
If $A(u)=0$, then
\[
 t^P=-\frac1{1+t}=\sigma^2(t).
\]
Since $K=\F_{p^{6s+3}}$, this implies
\[
 t\in\F_{p^{\gcd(6s+3,3s+1)}}=\F_p.
\]
Thus \eqref{eq:t-relation} becomes
\begin{equation}\label{eq:t-cubic}
 t^2+t+1=0.
\end{equation}
If $p\equiv2\pmod3$, there are no nonzero exceptional fibres, and
$M_3=(q^2-1)+q=q^2+q-1$.

If $p\equiv1\pmod3$, let $\omega\in\F_p$ be a primitive cube root
of unity.  By Hilbert's Theorem~90, the equations
$\sigma(u)=\omega u$ and $\sigma(u)=\omega^2u$ each have a nonzero
solution because the relative norms of $\omega$ and $\omega^2$ from
$K/\F_q$ are $1$; their nonzero solution sets are one-dimensional
$\F_q$-spaces.  Since $1+\omega+\omega^2=0$, they lie in
$V=\ker\Tr_{K/\F_q}$.  Thus the exceptional nonzero $u$ form precisely
these two $\F_q$-lines.  For such $u$, substituting
$y=u(1-t)/3$ in \eqref{eq:B-def} gives
\[
 B(y)=\frac{(1-t)^2}{3}u^{P+2}\ne0.
\]
Hence the $2(q-1)$ nonzero exceptional fibres contribute no solution;
the $(q-1)^2$ generic fibres contribute one solution each, and the
fibre $u=0$ contributes $q$.  Thus $M_3=q^2-q+1$.
Formula \eqref{eq:count3} follows from \eqref{eq:Z-Mn}.
\end{proof}

Write
\[
 L_{\cZ}(T)=\prod_j(1-\alpha_jT)
 =1+c_1T+c_2T^2+c_3T^3+\cdots
\]
and
\[
 S_n:=\sum_j\alpha_j^n
 =q^n+1-\#\cZ_{p,s}(\F_{q^n}).
\]
Put $A:=q(q-1)$.  Proposition~\ref{prop:three-counts} gives
\[
 S_1=S_2=-A,
 \qquad
 S_3=\varepsilon A,
\]
where
\[
 \varepsilon:=
 \begin{cases}
 1,&p\equiv1\pmod3,\\
 -1,&p\equiv2\pmod3.
 \end{cases}
\]
Newton's identities give
\begin{equation}\label{eq:c123}
 c_1=A,\qquad
 c_2=\frac{A(A+1)}2,\qquad
 c_3=\frac{A(A^2+3A-2\varepsilon)}6.
\end{equation}

\begin{theorem}\label{thm:slope-bracket}
Let $p>3$.  Then
\[
 \frac1{P+2}
 \le\lambda_{\min}(\cZ_{p,s})
 \le\frac13.
\]
\end{theorem}

\begin{proof}
For the upper bound, normalize the valuation by $v_q(q)=1$.  Since
$p>3$, both $2$ and $6$ are $p$-adic units.  Furthermore,
\[
 v_q(A)=1,\qquad A+1\equiv1\pmod p,
\]
and
\[
 A^2+3A-2\varepsilon\equiv-2\varepsilon\not\equiv0\pmod p.
\]
Hence
\[
 v_q(c_1)=v_q(c_2)=v_q(c_3)=1.
\]
The Newton polygon has ordinate at most $1$ at abscissa $3$, so the
sum of the first three slopes is at most $1$.  In particular
$\lambda_{\min}\le1/3$.

For the lower bound, the degree-$p$ quotients of $\cZ_{p,s}$ are the
curves $C_{[0:\mu]}$ with reduced polynomial of degree $P+2$.
Zhu's Newton-over-Hodge theorem \cite{Zhu} gives first slope at least
$1/(P+2)$ for every such quotient.  The elementary abelian
factorization of $\Jac(\cZ_{p,s})$ gives the same lower bound for
$\cZ_{p,s}$.
\end{proof}

\begin{corollary}\label{cor:nonss}
For $p>3$, the curves $\cZ_{p,s}$ and $\cR_{p,s}$ are not
supersingular.  Moreover, $\cR_{p,s}$ is neither maximal nor minimal
over any finite extension of $\F_q$.
\end{corollary}

\begin{proof}
Theorem~\ref{thm:slope-bracket} gives a slope strictly below $1/2$ in
$\Jac(\cZ_{p,s})$, and Theorem~\ref{thm:decomp} makes this an isogeny
factor of $\Jac(\cR_{p,s})$.  By \eqref{eq:maxmin-ss}, a curve that is maximal or minimal over a
finite extension has supersingular Jacobian.  Hence the slope below
$1/2$ obtained above rules out both possibilities.
\end{proof}

\begin{remark}\label{rem:p3}
For $p=3$, $\cR_{3,s}$ is the classical Ree curve.  Its Frobenius
polynomial is a product of powers of
\[
 T^2+q
 \qquad\text{and}\qquad
 T^2+3q_0T+q;
\]
see \cite{TT}.  Hence it is supersingular.  Thus, among the odd
characteristics in the present family, characteristic $3$ is exactly
the supersingular case.
\end{remark}

\section{The full geometric automorphism group}\label{sec:aut}

Assume throughout this section that $p>3$, and set
\[
 X:=\cR_{p,s}\times_{\F_q}\overline{\F}_q.
\]
We use the group $U$ defined in \eqref{eq:Gamma-size}.

\begin{lemma}\label{lem:big}
The pair $(X,U)$ is a big action.  In fact,
\begin{equation}\label{eq:U-over-g}
 \frac{|U|}{g_X}>q_0.
\end{equation}
\end{lemma}

\begin{proof}
Using $q=pq_0^2$ and $P=pq_0$, a direct calculation gives
\[
\begin{aligned}
 2q^3-q_0(q-1)\bigl[P+q(P+1)\bigr]
 =p q_0^2\bigl(p^2q_0^4-pq_0^3+q_0+1\bigr)>0.
\end{aligned}
\]
Together with \eqref{eq:genus-R}, this is exactly
\eqref{eq:U-over-g}.  Since $q_0\ge p\ge5$,
\[
 \frac{|U|}{g_X}>q_0>\frac{2p}{p-1},
\]
so $(X,U)$ is a big action.
\end{proof}

\begin{remark}\label{rem:GK-large-p}
The action also lies in the stronger ``large $p$-group'' range of
Giulietti--Korchm\'aros \cite[Theorem~1.1]{GKp}.  Indeed,
\[
 |U|>q_0g_X\ge pg_X>
 \frac{p^2}{p^2-p-1}(g_X-1).
\]
Since $X$ has $p$-rank zero, their theorem places this action in the
case of a unique totally ramified place and no other ramification,
exactly as in the explicit ramification calculation of
Proposition~\ref{prop:ramification}.
\end{remark}

\begin{lemma}\label{lem:derived}
Put
\[
 N:=\{u(0,c,d):c,d\in\F_q\}.
\]
Then
\[
 D(U)=N,
\]
so $|N|=q^2$ and $U/N\cong(\F_q,+)$ acts on $x$ by translations.
Moreover,
\[
 X/N\cong\mathbb P^1_x.
\]
\end{lemma}

\begin{proof}
Formula \eqref{eq:U-law} shows that $U/N$ is abelian, hence
$D(U)\subseteq N$.  Commuting an $x$-translation with a
$y$-translation gives all $z$-translations.  The $y$-coordinates of
commutators of two $x$-translations span the $\F_p$-space generated
by
\[
 b^{q_0}b'-{b'}^{q_0}b.
\]
We claim that this span is all of $\F_q$.  If $a\in\F_q$ is
orthogonal to it for the trace pairing, then the adjoint of the
$p^s$-Frobenius gives, for every $b\in\F_q$,
\[
 a b^{p^s}=a^{p^{s+1}}b^{p^{s+1}}.
\]
Taking $b=1$ gives $a=a^{p^{s+1}}$.  If $a\ne0$, varying $b$ would
force $b^{p^s}=b^{p^{s+1}}$ for every $b\in\F_q$, impossible because
$\F_q\ne\F_p$.  Thus $a=0$.  Hence all $y$-translations, and already
all $z$-translations, lie in $D(U)$.  Therefore $D(U)=N$.  The fixed
field of $N$ is $\overline{\F}_q(x)$, which proves the last assertion.
\end{proof}

Let
\[
 G:=\Aut_{\overline{\F}_q}(X).
\]
Since $U$ fixes $P_\infty$, we have
$U\le G_{P_\infty}^{(1)}$.  Lemma~\ref{lem:big} therefore gives
\[
 |G_{P_\infty}^{(1)}|\ge |U|>q_0g_X>
 \frac{p}{p-1}g_X.
\]
Theorem~1.3 of Giulietti--Korchm\'aros \cite{GK} now implies that
either $G$ fixes $P_\infty$, or $X$ belongs to one of their four
exceptional families.  The Suzuki and Ree cases are excluded because
$p>3$.  The Hermitian case is excluded because a Hermitian curve is
supersingular, whereas $X$ is not supersingular by
Corollary~\ref{cor:nonss}.  In the remaining odd hyperelliptic case the
hyperelliptic involution is central and has order prime to $p$, so the
nonabelian group $U$ would inject into the automorphism group of the
rational quotient.  This is impossible because every finite
$p$-subgroup of $\operatorname{PGL}_2$ in characteristic $p$ is
elementary abelian; see, for instance, \cite{Faber}.  Hence
\begin{equation}\label{eq:G-fixes-infty}
 G=G_{P_\infty}.
\end{equation}

Put
\[
 A_1:=G_{P_\infty}^{(1)}.
\]
Now Corollary~2.11 of Matignon--Rocher \cite{MR} applies to the big
action $(X,U)$ and its full wild inertia group $A_1$, giving
\begin{equation}\label{eq:A1-derived}
 D(A_1)=D(U)=N.
\end{equation}

By \eqref{eq:A1-derived}, the group $A_1/N$ acts on
$X/N\cong\mathbb P^1_x$ by translations.  Thus
\begin{equation}\label{eq:E-space}
 A_1/N=\{x\mapsto x+\beta:\beta\in E\}
\end{equation}
for a finite $\F_p$-vector space
$E\subset\overline{\F}_q$ containing $\F_q$.

\begin{lemma}\label{lem:E=Fq}
One has $E=\F_q$.  Consequently $A_1=U$.
\end{lemma}

\begin{proof}
The degree-$p$ quotient characters of $N$ are represented by
\eqref{eq:F-lambda-mu}.  Their conductor exponents are $P+2$ for
$\mu=0$ and $P+3$ for $\mu\ne0$.  Hence
\[
 \mathcal W:=\{(\lambda,0):\lambda\in\F_q\}
\]
is intrinsically characterized as the character subspace of conductor
at most $P+2$.

Let $\beta\in E$, and choose $\sigma\in A_1$ inducing the
translation $x\mapsto x+\beta$ on $X/N$.  Since
$N=D(A_1)$ by \eqref{eq:A1-derived}, the subgroup $N$ is
characteristic in $A_1$.  Hence conjugation by $\sigma$ permutes the
degree-$p$ quotients of $X\to X/N$ and preserves their Artin
conductors, so it stabilizes the projective character subspace
corresponding to
\[
 \mathcal W:=\{(\lambda,0):\lambda\in\F_q\}.
\]
Accordingly, for each $\lambda\in\F_q^\times$ there exist
$\lambda'\in\F_q^\times$ and $a\in\F_p^\times$ such that, after
pullback by $\sigma$, the associated Artin--Schreier classes satisfy
\begin{equation}\label{eq:AS-line-beta}
 [F_{\lambda,0}(x+\beta)]
   =a[F_{\lambda',0}(x)]
 \qquad\text{in }\overline{\F}_q(x)/\wpmap(\overline{\F}_q(x)).
\end{equation}
Here
\[
 F_{\lambda,0}(x)
 =\lambda^{1/q_0}x^{P+1}-\lambda x^{q_0+1}.
\]
Translation does not change the coefficient
$\lambda^{1/q_0}$ of $x^{P+1}$, and Artin--Schreier reduction cannot
alter this term because $P+1$ is prime to $p$.  Comparing the leading
coefficients in \eqref{eq:AS-line-beta} gives
\[
 \lambda^{1/q_0}=a({\lambda'})^{1/q_0}.
\]
Since $a\in\F_p^\times$, raising to the $q_0$-th power yields
$\lambda=a\lambda'$.  Therefore
$aF_{\lambda',0}=F_{\lambda,0}$, and
\eqref{eq:AS-line-beta} becomes
\[
 [F_{\lambda,0}(x+\beta)]=[F_{\lambda,0}(x)].
\]
To make the reduction explicit, in characteristic $p$ we have
\[
 (x+\beta)^{P+1}-x^{P+1}
   =\beta x^P+\beta^P x+\beta^{P+1},
\]
and
\[
 (x+\beta)^{q_0+1}-x^{q_0+1}
   =\beta x^{q_0}+\beta^{q_0}x+\beta^{q_0+1}.
\]
Modulo $\wpmap(\overline{\F}_q(x))$, a monomial $c x^{p^r}$ reduces
to $c^{1/p^r}x$, while constants may be normalized away.  Therefore
$F_{\lambda,0}(x+\beta)-F_{\lambda,0}(x)$ reduces to a linear
polynomial whose coefficient is
\begin{equation}\label{eq:beta-linear}
 \lambda\bigl(\beta^{1/P}-\beta^{q_0}\bigr)
 +\lambda^{1/q_0}
 \bigl(\beta^P-\beta^{1/q_0}\bigr).
\end{equation}
This must vanish for every $\lambda\in\F_q$.  If the two coefficients
in \eqref{eq:beta-linear} were not both zero, the nonzero polynomial
$AX+BX^P\in\overline{\F}_q[X]$ would vanish on all $q$ elements of
$\F_q$.  This is impossible because $P<q$ for $s\ge1$.  Therefore
\[
 \beta^{1/P}=\beta^{q_0}.
\]
Raising to the $P$-th power and using $Pq_0=q$ gives
$\beta^q=\beta$.  Hence $\beta\in\F_q$, proving $E=\F_q$.
The kernel of the action $A_1\to\Aut(X/N)$ is
$A_1\cap\Gal(\overline{\F}_q(X)/\overline{\F}_q(x))=N$.  Therefore
\[
 |A_1|=|N|\,|E|=q^2q=q^3=|U|.
\]
Since $U\subseteq A_1$, it follows that $A_1=U$.
\end{proof}

\begin{theorem}\label{thm:Aut}
For $p>3$,
\[
 \Aut_{\overline{\F}_q}(\cR_{p,s})
 =\Aut_{\F_q}(\cR_{p,s})
 =\Gamma
 \cong U\rtimes\F_q^\times.
\]
In particular,
\[
 |\Aut_{\overline{\F}_q}(\cR_{p,s})|=q^3(q-1).
\]
\end{theorem}

\begin{proof}
We have shown that $G=G_{P_\infty}$ and that its first ramification
group is
\[
 G_{P_\infty}^{(1)}=A_1=U.
\]
The first ramification group is the unique Sylow $p$-subgroup of the
point stabilizer; hence, since $G=G_{P_\infty}$, the subgroup $U$ is
normal in $G$.  Therefore its derived subgroup
$N=D(U)$, which is characteristic in $U$, is also normal in $G$.
Consequently $G/N$ acts faithfully on
$X/N\cong\mathbb P^1_x$ and fixes infinity.  Every induced map has the
form
\[
 x\longmapsto ax+b.
\]
It normalizes the translation subgroup
$U/N\cong(\F_q,+)$, so $a\F_q=\F_q$ and hence
$a\in\F_q^\times$.  The corresponding scaling already lies in
$\Gamma$.  After composing by its inverse, the remaining image in
$G/N$ is the translation $t_b:x\mapsto x+b$.  If $b\ne0$, then
$\langle t_b\rangle$ has order $p$.  Its full preimage in $G$ is an
extension of the $p$-group $N$ by a cyclic group of order $p$, hence
is itself a $p$-group.  Since $U$ is the unique Sylow $p$-subgroup of
$G$, that preimage is contained in $U$.  Thus $t_b\in U/N$, which
forces $b\in\F_q$; the case $b=0$ is immediate.  Consequently
\[
 |G/N|\le q(q-1),
\]
and therefore $|G|\le q^3(q-1)$.  Equality is attained by $\Gamma$.
Thus $G=\Gamma$, and every element is defined over $\F_q$.
\end{proof}

\begin{corollary}\label{cor:aut-growth}
Let $p>3$, and write $g:=g(\cR_{p,s})$.  Then
\begin{equation}\label{eq:aut-linear-lower}
 \frac{|\Aut(\cR_{p,s})|}{g}>q q_0,
 \qquad
 \frac{|U|}{g}>q_0.
\end{equation}
In particular,
\[
 |\Aut(\cR_{p,s})|>84(g-1).
\]
For fixed $p$ and $s\to\infty$,
\begin{align}
 |\Aut(\cR_{p,s})|
 &\sim 2^{8/5}p^{-4/5}\,g^{8/5},\label{eq:aut-asymp}\\
 |U|
 &\sim 2^{6/5}p^{-3/5}\,g^{6/5}.\label{eq:U-asymp}
\end{align}
\end{corollary}

\begin{proof}
The second inequality in \eqref{eq:aut-linear-lower} is
Lemma~\ref{lem:big}.  Since
\[
 \frac{|\Aut(\cR_{p,s})|}{g}
 =\frac{2q^3}{P+q(P+1)},
\]
the first follows from
\[
 2q^3-q q_0\bigl[P+q(P+1)\bigr]
 =p^2q_0^4(pq_0^2-q_0-1)>0.
\]
As $q q_0\ge625$, the comparison with the characteristic-zero
Hurwitz quantity $84(g-1)$ follows.

Finally, with $p$ fixed and $q_0\to\infty$,
\[
 g\sim\frac{p^3}{2}q_0^5,\qquad
 |\Aut(\cR_{p,s})|\sim p^4q_0^8,\qquad
 |U|=p^3q_0^6.
\]
Eliminating $q_0$ gives \eqref{eq:aut-asymp} and
\eqref{eq:U-asymp}.
\end{proof}

\begin{remark}
For $p=3$, the same subgroup $\Gamma$ is the stabilizer of
$P_\infty$ in the full Ree group.  The latter has order
$q^3(q-1)(q^3+1)$; see Hansen--Pedersen \cite{HP}.  Thus the extra
factor $q^3+1$ in the automorphism group occurs together with the
supersingularity phenomenon only in characteristic $3$.  It also
changes the asymptotic scale dramatically: for the classical Ree
family the full group has order $\Theta(g^{14/5})$, whereas
Corollary~\ref{cor:aut-growth} gives $\Theta(g^{8/5})$ for every fixed
odd $p>3$.
\end{remark}

\section{An application of Voloch's construction}\label{sec:Voloch}

We finish with an application of Voloch's construction to unramified
abelian covers over $\F_{q^2}$ in which all $\F_{q^2}$-rational points of
$\cR_{p,s}$ split completely.  The following elementary observation is the
key.

\begin{proposition}\label{prop:Fq2=same}
For every odd prime $p$,
\[
 \cR_{p,s}(\F_{q^2})=\cR_{p,s}(\F_q).
\]
In particular,
\[
 \#\cR_{p,s}(\F_{q^2})=q^3+1.
\]
\end{proposition}

\begin{proof}
Let $(x,y,z)$ be an affine $\F_{q^2}$-rational point and set
$d:=x^q-x$.  Solvability of the first equation over $\F_{q^2}$ implies
\[
 0=\Tr_{\F_{q^2}/\F_q}(x^{q_0}d).
\]
Since $d^q=-d$,
\[
 \Tr_{\F_{q^2}/\F_q}(x^{q_0}d)
 =d\bigl(x^{q_0}-(x+d)^{q_0}\bigr)
 =-d^{q_0+1}.
\]
Thus $d=0$ and $x\in\F_q$.  Both right-hand sides in
\eqref{eq:family} are then zero, so $y,z\in\F_q$.  The point at
infinity is already rational.
\end{proof}

Let $J:=\Jac(\cR_{p,s})$ and write $g:=g(\cR_{p,s})$.  Embed
$\cR_{p,s}$ in $J$ using $P_\infty$ as the origin, and let $\pi_q$
denote the $q$-Frobenius endomorphism of $J$.  Proposition~\ref{prop:Fq2=same}
gives
\[
 \cR_{p,s}(\F_{q^2})=\cR_{p,s}(\F_q)
 \subset J(\F_q)\subset J(\F_{q^2}).
\]
The specialization of Voloch's construction \cite{Voloch} to the
subgroup $J(\F_q)$ is represented by the isogeny
\[
 \phi:=1+\pi_q:J\longrightarrow J.
\]
Indeed, over $\F_{q^2}$ one has
$1-\pi_q^2=(1+\pi_q)(1-\pi_q)$ and
$\ker(1-\pi_q)=J(\F_q)$.

\begin{theorem}\label{thm:Voloch-cover}
There exists a connected unramified abelian cover
\[
 \widetilde{\cR}_{p,s}\longrightarrow\cR_{p,s}
\]
defined over $\F_q$ (and hence over $\F_{q^2}$) of degree
\begin{equation}\label{eq:nVoloch}
 n:=[J(\F_{q^2}):J(\F_q)]
 =L_{\cR}(-1)
 =L_{\cY}(-1)L_{\cZ}(-1)^q.
\end{equation}
Every $\F_{q^2}$-rational point of $\cR_{p,s}$ splits completely.
Consequently
\begin{align}
 g(\widetilde{\cR}_{p,s})&=n(g-1)+1,\label{eq:cover-genus}\\
 \#\widetilde{\cR}_{p,s}(\F_{q^2})&\ge n(q^3+1).\label{eq:cover-points}
\end{align}
Moreover $n>1$.
\end{theorem}

\begin{proof}
Consider the separable isogeny
\[
 \phi:=1+\pi_q:J\longrightarrow J.
\]
It is separable because $d\pi_q=0$, so $d\phi=\mathrm{id}$.  Let
$\widetilde{\cR}_{p,s}$ be the pullback of the Abel--Jacobi embedding
$\cR_{p,s}\hookrightarrow J$ along $\phi$.  This is precisely the
Voloch construction for the subgroup $J(\F_q)$ over $\F_{q^2}$.
The pullback is the connected finite \emph{\'etale} abelian cover
appearing in Voloch's construction; see \cite[pp.~756--757]{Voloch},
where the pullback under the corresponding isogeny is shown to have
degree equal to the index of the chosen subgroup.  Hence the degree of
the present cover is $\deg\phi$.

For any prime $\ell\ne p$,
\[
 \deg(1+\pi_q)
 =\det(1+\pi_q\mid T_\ell J)
 =\prod_{i=1}^{2g}(1+\alpha_i)
 =L_{\cR}(-1).
\]
On the other hand,
\[
 \frac{\#J(\F_{q^2})}{\#J(\F_q)}
 =\frac{\prod_i(1-\alpha_i^2)}{\prod_i(1-\alpha_i)}
 =\prod_i(1+\alpha_i)
 =L_{\cR}(-1),
\]
which gives the first two equalities in \eqref{eq:nVoloch}; the last
one follows from \eqref{eq:L-decomp}.

It remains to check complete splitting.  Let
$Q\in\cR_{p,s}(\F_{q^2})=\cR_{p,s}(\F_q)$ and let $z\in J(\overline{\F}_q)$
satisfy
\[
 z+\pi_q(z)=Q.
\]
Applying $\pi_q$ and using $\pi_q(Q)=Q$ gives
\[
 \pi_q(z)+\pi_q^2(z)=Q.
\]
Subtracting the two equalities yields $\pi_q^2(z)=z$.  Thus every
point of the fibre $\phi^{-1}(Q)$ is $\F_{q^2}$-rational, so $Q$
splits completely.  Hurwitz now gives \eqref{eq:cover-genus}, and
complete splitting gives \eqref{eq:cover-points}.

Finally, $|\alpha_i|=\sqrt q$ for every complex embedding, so
\[
 n=L_{\cR}(-1)
 =\prod_i|1+\alpha_i|
 \ge(\sqrt q-1)^{2g}>1.
\]
\end{proof}

\begin{remark}
Voloch used the same construction for Hermitian and Suzuki curves and
noted analogous examples coming from Ree curves \cite{Voloch}.
Borges--Coutinho \cite[Section~6]{BC} applied it to the generalized
Suzuki curve, which is the quotient $\cY_{p,s}$ identified in
Remark~\ref{rem:Y=GS}.  The point here is therefore the decomposition
for the two-equation curve: the degree separates into the already known
$\cY$-contribution and the complementary $\cZ$-contribution in
\eqref{eq:nVoloch}.
\end{remark}

\section{Final remarks}

The results above separate three phenomena that coincide for the
classical Ree curve but need not coincide in odd characteristic.
First, the ray-class structure forces $p$-rank zero throughout the
family.  Second, the quotient $\cY_{p,s}$ is the already known
generalized Suzuki curve of Borges--Coutinho and is geometrically
supersingular.  Third, as soon as $p>3$, the complementary
$\cZ$-factor acquires a slope strictly below $1/2$.  Thus
$\cR_{p,s}$ provides a family of Castle curves with $p$-rank zero and
a large automorphism group, but with a genuinely non-supersingular
Newton polygon.

This profile is complementary to the elementary abelian covers
studied by Borges--Fukasawa \cite{BF}, where positive $p$-rank and
large Sylow $p$-subgroups are central.  Here the $p$-rank vanishes,
yet
\[
 \frac{|U|}{g}\to\infty,\qquad
 |\Aut(\cR_{p,s})|=\Theta(g^{8/5}),
\]
and the full automorphism group is still determined explicitly.

The decomposition \eqref{eq:J-decomp} also identifies the remaining
arithmetic problem sharply.  The first slope of the new factor lies
in the nontrivial interval
\[
 \frac1{P+2}\le\lambda_{\min}(\cZ_{p,s})\le\frac13.
\]
The point-count method of Section~\ref{sec:NP} admits an exact
formulation.  From \eqref{eq:Z-Mn}, for every $n\ge2$,
\begin{equation}\label{eq:Sn-Mn-final}
 S_n=q\bigl(q^{n-1}-M_n\bigr),
\end{equation}
so
\[
 v_q(S_n)=1\quad\Longleftrightarrow\quad p\nmid M_n.
\]
Because Proposition~\ref{prop:prank} gives $p$-rank zero, every
nonconstant coefficient of $L_{\cZ}(T)$ has positive $q$-adic
valuation.  The following single-index criterion sharpens the form of
the point-count argument used in Section~\ref{sec:NP}.

\begin{proposition}\label{prop:single-index}
Let $2\le N\le 2g(\cZ_{p,s})$.  If $p\nmid N$ and $p\nmid M_N$, then
\[
 v_q(c_N)=1,
 \qquad\text{and hence}\qquad
 \lambda_{\min}(\cZ_{p,s})\le\frac1N.
\]
\end{proposition}

\begin{proof}
Newton's identity in degree $N$ is
\begin{equation}\label{eq:Newton-final}
 N c_N+\sum_{i=1}^{N} c_{N-i}S_i=0,
 \qquad c_0:=1.
\end{equation}
For every $i\ge2$, equation~\eqref{eq:Sn-Mn-final} gives
$v_q(S_i)\ge1$, while $S_1=-q(q-1)$ has valuation exactly $1$.
If $i<N$, then $N-i\ge1$, and the $p$-rank-zero hypothesis gives
$v_q(c_{N-i})>0$.  Therefore
\[
 v_q(c_{N-i}S_i)>1\qquad (i<N).
\]
On the other hand, since $N\ge2$, the assumption $p\nmid M_N$ and
\eqref{eq:Sn-Mn-final} imply $v_q(S_N)=1$.  Thus $S_N$ is the unique
summand on the right-hand side of \eqref{eq:Newton-final} having
valuation $1$.  Consequently $v_q(Nc_N)=1$.  Since $p\nmid N$, one has
$v_q(N)=0$, and hence $v_q(c_N)=1$.  Therefore the Newton polygon has
ordinate at most $1$ at abscissa $N$.  The sum of its first $N$ slopes
is at most $1$, and hence its first slope is at most $1/N$.
\end{proof}

Thus a single nondivisibility condition $p\nmid M_N$ gives an upper
bound whenever $p\nmid N$; no information about the preceding values
$M_2,\ldots,M_{N-1}$ is needed.  On the other hand, the Hodge lower
bound gives
\[
 v_q(S_n)\ge n\lambda_{\min}\ge\frac{n}{P+2}.
\]
In particular $p\mid M_n$ for every $n>P+2$.  Since
$p\nmid(P+1)$ and $p\nmid(P+2)$, Proposition~\ref{prop:single-index}
yields the conditional implications
\[
 p\nmid M_{P+1}
 \quad\Longrightarrow\quad
 \frac1{P+2}\le\lambda_{\min}(\cZ_{p,s})\le\frac1{P+1},
\]
and
\[
 p\nmid M_{P+2}
 \quad\Longrightarrow\quad
 \lambda_{\min}(\cZ_{p,s})=\frac1{P+2}.
\]
It remains open whether these nondivisibility conditions occur in
general.  In particular, the second one would force the first slope
to attain the Hodge lower bound.

A second direction is to replace the second equation in
\eqref{eq:family} by
\[
 z^q-z=x^{iq_0}(x^q-x),\qquad 2\le i\le p-1.
\]
These are the higher members of the same ray-class pattern.  Their
conductor jumps suggest a filtration parallel to
Proposition~\ref{prop:ramification}, while the character-orbit
argument of Section~\ref{sec:decomp} suggests corresponding Jacobian
decompositions.  A systematic determination of their Newton polygons
would place the characteristic-three exceptional phenomenon in a
broader ray-class framework.

\section*{Acknowledgements}

The second author acknowledges partial financial support from CNPq,
grant no.~302774/2025-4, and FAPESP, grant no.~2024/00923-6.

\section*{Declaration on the use of AI-assisted tools}
During the preparation of this manuscript, the authors used
AI-assisted tools solely for English-language editing and stylistic
polishing.  The authors reviewed the resulting text, independently
verified the mathematical statements, proofs, computations and
references, and take full responsibility for the content of the
manuscript.

\end{document}